\documentclass[11pt]{article}%
\usepackage{amsfonts}
\usepackage{amsmath}
\usepackage{amssymb}
\usepackage{graphicx}
\usepackage[left=2.8cm,top=3.0cm,right=2.8cm,bottom=3 cm]{geometry}%
\newtheorem{theorem}{Theorem}

\newtheorem{definition}[theorem]{Definition}

\newtheorem{lemma}[theorem]{Lemma}

\newtheorem{proposition}[theorem]{Proposition}

\newenvironment{proof}[1][Proof]{\noindent\textbf{#1.} }{\ \rule{0.5em}{0.5em}}
\begin{document}

\title{Applying the Method of Moments to Examine Blow-Up Phenomena in KS Models with Variable Chemotactic Signals}
\author{Valeria Cuentas \& Elio Espejo \thanks{Corresponding author. \newline E-mail addresses: valeria.cuentasrodriguez@nottingham.edu.cn (V. Cuentas), elio-eduardo-espejo.arenas@nottingham.edu.cn(E. Espejo)}}
\date{}
\maketitle

\begin{abstract}
Cells encounter a diverse array of physical and chemical signals as they navigate their natural surroundings. However, their response to the simultaneous presence of multiple cues remains elusive. Particularly, the impact of topography alongside a chemotactic gradient on cell migratory behavior remains insufficiently explored. In this paper, we investigate the effects of topographical obstacles during chemotaxis. Our approach involves modifying the Keller-Segel model, incorporating a spatially dependent coefficient of chemotaxis. Through our analysis, we demonstrate that this coefficient plays a crucial role in preventing blow-up phenomena in cell concentration.

\bigskip

2000 Mathematics Subject Classification: 35K15, 35K55, 35Q60; Secondary 78A35

\end{abstract}

\section{Introduction}

Directed, single-cell migration is driven by external guidance cues, such as chemical, electrical, temperature, stiffness, and topographical gradients (cf. \cite{Cohen,Charras,Cuchi,Petrie,Poff,Rodriguez}). Natural cell environments often exhibit several such cues simultaneously. In the human body, processes occurring in multicue environments include immune response, cancer metastasis, and tissue regeneration. As of yet, it is unclear how external guidance cues relate to each other for various cell types and environments. Cells may ignore certain stimuli in favor of other cues or different cues might add up in affecting cell movement. In particular, contemporary research has illuminated the intricate relationship between chemotaxis and topography, shedding light on the influence of topographical cues on cellular chemotactic responses. Investigations have unveiled that the impact of topographical cues persists throughout cellular chemotaxis, with studies indicating that the topographical cue conserves its significance, contributing to the overall chemotactic effect (cf. \cite{Wondergem}).

In this paper, we focus on conditions that predict or prevent cell aggregation when obstacles interfere during the process. To this end, we propose to study the Keller-Segel-type model in a bounded domain $\Omega\subset\mathbb{R}^{n}$ given by:
\begin{equation}
\begin{cases}
\frac{\partial u}{\partial t}(x,t) = \operatorname{div}\left[ \nabla u(x,t) - \chi(x) u(x,t) \nabla v(x,t) \right], \quad t>0, x\in\Omega,\\
-\Delta v(x,t) = u(x,t), \quad t>0, x\in\Omega,\\
u(x,0) = u_{0}(x), \quad x\in\Omega.
\end{cases}
\label{eq}
\end{equation}

This model is complemented by the nonlinear no-flux condition:
\begin{equation}
\frac{\partial u(x,t)}{\partial n}-\chi(x)u(x,t)\frac{\partial v(x,t)}{\partial n}=0,
\label{initi}
\end{equation}

where $n$ denotes the outward unit normal vector to the $C^{1+\varepsilon}$ ($\varepsilon>0$) boundary $\partial\Omega$. Here, $u$ represents the cell density, $v$ denotes the chemoattractant, and $\chi(x)$ represents the chemical response influenced by a topographical cue. For the potential $v$, we assume:
\begin{equation}
v=K_{n}\ast u,
\label{convo}
\end{equation}

where $K_{n}$ is the fundamental solution of the $n$-dimensional Laplacian, namely, $K_{n}(x):=-\left\vert x\right\vert ^{2-n}/(\sigma_{n}(n-2)),$ $n\geq3,$where $\sigma_{n}$ is the volume of the unit $n-$ball and $K_{2}(x)=-\frac{1}{2\pi}\log\left\vert x\right\vert $. The initial-boundary value problem is supplemented with the initial condition:
\begin{equation}
u(x,0)=u_{0}(x)\geq0.
\label{positiv}
\end{equation}

The moment and mass be defined by $m(t):=\int u(x,t)\left\vert x\right\vert ^{2}dx$ and $M:=\int u_{0}dx=\int u(x,t)dx.$

For any arbitrary bounded smooth domains in $\mathbb{R}^{2}$ or in $\mathbb{R}^{3},$ the local-in-time existence of solutions in $L^{2}(\Omega)$ can be deduced from the proof provided in Theorem 1 of the reference \cite{BilerLocal}. Although this Theorem specifically addresses the scenario where $\chi(x)$ is constant, its existence argument seamlessly extends to our situation by assuming $\chi(x)\in L^{\infty}(\Omega)$. This argument relies on a standard application of the Schauder fixed-point theorem within an appropriate space of vector-valued functions in $L^{2}(\Omega)$.

\begin{definition}[Weak solution]
In the context of the problem (\ref{eq})-(\ref{positiv}) defined on $\Omega\times(0,T)$, weak $H^{1}(\Omega)$ solutions are understood as functions $u\in L^{\infty}\left( (0,T);L^{2}(\Omega)\right) \cap L^{2}\left( (0,T);H^{1}(\Omega)\right) $ which satisfy, for every test function $\eta\in H^{1}(\Omega\times(0,T))$ and for a.e. $t\in(0,T)$, the integral identity
\[
\int_{\Omega}u(x,t)\eta(x,t)dx-\int_{0}^{t}\int_{\Omega}u\eta_{t}+\int_{0}^{t}\int_{\Omega}(\nabla u+\chi(x)u\nabla v)\cdot\nabla\eta=\int_{\Omega}u_{0}(x)\eta(x,0)dx.
\]
Moreover, we require that for a.e. $t\in(0,T),v(\cdot,t)$ is a weak solution of (\ref{eq}) with
\[
v\in H^{1}(\Omega)\text{ with }\quad v=K_{n}\ast u.
\]
\end{definition}

\begin{theorem}[Local Existence]
Let $\Omega$ be a bounded domain in $\mathbb{R}^{n}$ with a boundary of class $C^{1+\varepsilon}$, where $\varepsilon>0$.

\begin{enumerate}
\item[(i)] For dimension $n=2$ or $n=3$, and initial data $0\leq u_{0}\in L^{2}(\Omega)$, there exists $T=T\left( \left\vert u_{0}\right\vert _{2}\right) $ such that the problem (\ref{eq})-(\ref{initi}) admits a unique weak solution $u\in L^{\infty}\left( (0,T);L^{2}(\Omega)\right) \cap L^{2}\left( (0,T);H^{1}(\Omega)\right) $. Additionally, $u_{t}\in L^{2}\left( (0,T);H^{-1}(\Omega)\right) $, $u(x,t)\geq0$ for almost every $x\in\Omega$ and $t\geq0$, and $\int_{\Omega}u(x,t)dx=\int_{\Omega}u_{0}(x)dx$.

\item[(ii)] For dimension $n\geq2$ and $0\leq u_{0}\in L^{p}(\Omega)$ with $p>n/2$, there exists $T=T\left( p,\left\vert u_{0}\right\vert _{p}\right) >0$ and a weak solution $u$ such that $u\in L^{\infty}\left( (0,T);L^{p}(\Omega)\right) $ and $u^{p/2}\in L^{2}\left( (0,T);H^{1}(\Omega)\right) $.
\end{enumerate}

These solutions are unique when $p>n$, and regular when $p>n/2$ in the sense that $u\in L_{loc\text{ }}^{\infty}\left( (0,T);L^{\infty}(\Omega)\right) $.
\end{theorem}

\begin{proof}
The proof follows the same argument of \cite[Theorem 1, Proposition 1]{BilerLocal} with minor modifications.
\end{proof}

\section{The role of topography during cell aggregation}

Let us denote by $\chi:\mathbb{R}^{2}\rightarrow\mathbb{R}^{+}$ a positive, smooth function that increases radially, i.e.,
\begin{equation}
\chi(x)\geq\chi(y)\text{ \ if \ \ }\left\vert x\right\vert \geq\left\vert y\right\vert .
\label{RadialNondecrea}
\end{equation}
An example of this kind of function can be constructed by choosing an increasing bounded smooth function $f:\mathbb{R}\rightarrow\mathbb{R}^{+}$ and defining $\chi(x):=f(\left\vert x\right\vert ^{2}).$ Specific examples are $\chi(x)=\frac{\left\vert x\right\vert ^{2}}{1+\left\vert x\right\vert ^{2}}+1$ and $\arctan\left\vert x\right\vert ^{2}$. One example of such a function without being radially symmetric is
\[
\chi(x_{1},x_{2})=\left\{
\begin{array}
[c]{cc}%
x_{1}^{2}/\left\vert x\right\vert , & \text{for }(x_{1},x_{2})\neq(0,0)\\
0, & \text{at the point }(0,0)
\end{array}
\right.
\]
We call $\Omega$ a star-shaped domain if there exists $x_{0}\in\mathbb{R}^{n}$ such that
\[
(x-x_{0})\cdot\nu\geq0\text{ for all }x\in\partial\Omega,
\]
where $\nu$ is the unit outward normal to $\partial\Omega$ at $x,$ cf. \cite{Pohoazev}.

\begin{theorem}[Blow-up in dimension two]
If $\Omega\subset\mathbb{R}^{2}$, is a star-shaped domain with respect to $0\in\Omega$, and $\chi:\mathbb{R}^{2}\rightarrow\mathbb{R}^{+}$ is a function that satisfies the monotonicity condition (\ref{RadialNondecrea}) and $\chi(\mathbf{0})>0$, then for initial data satisfying $\int_{\Omega}u_{0}dx=:M>\frac{8\pi}{\chi(0)}$, there are no global solutions to (\ref{eq})-(\ref{positiv}).
\end{theorem}

\begin{proof}
To simplify, we provide a formal argument demonstrating that the second moment $\int_{\Omega}u|\mathbf{x}|^{2}\,dx$ becomes negative in a finite amount of time. The computations below can be justified by writing the integral version of the corresponding differential inequalities.
First, we observe that the cell-density $u$ satisfies:
\begin{align*}
\frac{d}{dt}\int_{\Omega}u|\mathbf{x}|^{2}\,dx  = & -2\int_{\Omega}\nabla u\cdot\nabla x\,dx + 2\chi\int_{\Omega}x\cdot(u\chi(x)\nabla(K_{2}\ast u))\,dx \\
 = & -2\int_{\partial\Omega}u(x\cdot\nu)\,dx + 4\int_{\Omega}u\,dx + 2\chi\int_{\Omega}x\cdot u\chi(x)\int_{\Omega}\frac{-1}{2\pi}\frac{x-y}{|x-y|^{2}}u\,dy\,dx \\
 = & 4M - \frac{\chi}{\pi}I.
\end{align*}
Since $\Omega$ is a star-shaped domain with respect to $0$, we have $x\cdot\nu\geq0$ on $\partial\Omega,$ thus
\begin{align*}
\frac{d}{dt}\int_{\Omega}u|\mathbf{x}|^{2}\,dx \leq & 4\int_{\Omega}u\,dx + 2\chi\int_{\Omega}x\cdot u\chi(x)\int_{\Omega}\frac{-1}{2\pi}\frac{x-y}{|x-y|^{2}}u\,dy\,dx \\
 \leq & 4\int_{\Omega}u_{0}\,dx + 2\chi\int_{\Omega}x\cdot u\chi(x)\int_{\Omega}\frac{-1}{2\pi}\frac{x-y}{|x-y|^{2}}u\,dy\,dx
\end{align*}
Next, we interchange $x$ and $y$ in the integral $I$ to get
\[
I = -\frac{1}{2}\int_{\Omega\times\Omega}\left(  y\cdot\chi(y)\frac{x-y}{\left\vert x-y\right\vert ^{2}}u(x,t)u(y,t)dy\right)  dx,
\]
and hence,
\[
\frac{d}{dt}\int_{\Omega}u|\mathbf{x}|^{2}\,dx = 4M - \frac{\chi}{2\pi}\int_{\Omega\times\Omega}[\chi(x)x-\chi(y)y]\cdot\frac{x-y}{\left\vert x-y\right\vert ^{2}}u(x,t)u(y,t)dydx.
\]
A main difficulty arising at this point lies in estimating the last integral. To address this, we observe that
\[
2[\chi(x)x-\chi(y)y]\cdot\frac{x-y}{\left\vert x-y\right\vert ^{2}} = \chi(x) + \chi(y) + \frac{|\mathbf{x}|^{2}-|\mathbf{y}|^{2}}{|x-y|^{2}}[\chi(x)-\chi(y)].
\]
which can be straightforwardly verified by expanding and comparing the expressions arising on each side of the equivalent identity
\[
2[\chi(x)x-\chi(y)y]\cdot\left(  x-y\right)  = \left(  \chi(x)+\chi(y)\right)\left\vert x-y\right\vert ^{2} + \left(  |\mathbf{x}|^{2}-|\mathbf{y}|^{2}\right)  [\chi(x)-\chi(y)].
\]
Next, we apply the monotonicity property (\ref{RadialNondecrea}) to obtain
\[
[\chi(x)x-\chi(y)y]\cdot\frac{x-y}{\left\vert x-y\right\vert ^{2}} \geq \frac{1}{2}\chi(x) + \frac{1}{2}\chi(y) \geq \chi(\mathbf{0}),
\]
leading to the key estimate
\begin{align*}
\frac{d}{dt}\int_{\Omega}u|\mathbf{x}|^{2}\,dx \leq & 4M - \frac{\chi(\mathbf{0})}{2\pi}\int_{\Omega\times\Omega}u(x,t)u(y,t)dydx \\
 = & 4M - \frac{\chi(\mathbf{0})}{2\pi}M^{2}.
\end{align*}
Considering the assumption $\chi(0)>0$, we conclude that the second moment becomes negative in a finite amount of time if
\[
M > \frac{8\pi}{\chi(\mathbf{0})},
\]
which is absurd since $u$ remains nonnegative.
\end{proof}

\begin{theorem}[Blow-up in dimension $n\geq3$]
Let $\Omega\subset\mathbb{R}^{n}$ be a star-shaped domain with respect to $0\in\Omega$ and $p$ a constant satisfying $2\leq p\leq n$. Assume that the initial data satisfies $\int_{\Omega}u_{0}dx=:M>\frac{2^{n}n\sigma_{n}}{\chi}.$ Then the system
\[
\begin{array}
[c]{cc}%
\partial_{t}u=\Delta u-\chi\nabla\cdot(\left\vert x\right\vert ^{p-2}u\nabla v), & x\in\mathbb{R}^{n},t>0,\\
-\Delta v=u,\text{ }v(x,t)=\frac{-1}{\sigma_{n}}\int u(y,t)\left\vert x-y\right\vert ^{2-n}dy & x\in\mathbb{R}^{n},t>0,\\
u(x,0)=u_{0}(x)\geq0, & x\in\mathbb{R}^{n},
\end{array}
\]
does not have global solutions.
\end{theorem}

\begin{proof}
We apply the moments' technique as follows.
\begin{align*}
\frac{d}{dt}\int_{\Omega}u\left\vert x\right\vert ^{2}dx = & -2\int_{\Omega}x\cdot\nabla udx + 2\chi\int_{\Omega}x\cdot\left(  \left\vert x\right\vert ^{p-2}u\nabla v\right)  dx \\
 = & 4\int_{\Omega}udx + 2\chi\int_{\Omega}x\cdot\left(  \left\vert x\right\vert ^{p-2}u\nabla(K\ast u)\right)  dx \\
 = & 4\int_{\Omega}u_{0}dx - \frac{2\chi}{n\sigma_{n}}\int_{\Omega}x\cdot\left\vert x\right\vert ^{p-2}u\left(  \int_{\Omega}\frac{x-y}{\left\vert x-y\right\vert ^{n}}u(y,t)dy\right)  dx \\
 = & 4M - \frac{2\chi}{n\sigma_{n}}\int_{\Omega\times\Omega}\left(  x\left\vert x\right\vert ^{p-2}\cdot\frac{x-y}{\left\vert x-y\right\vert ^{n}}u(x,t)u(y,t)dy\right)  dx \\
 = & 4M - \frac{2\chi}{n\sigma_{n}}I.
\end{align*}
We interchange $x$ and $y$ in the integral $I$ to get
\[
I = -\frac{1}{2}\int_{\Omega\times\Omega}\left(  y\left\vert y\right\vert ^{p-2}\cdot\frac{x-y}{\left\vert x-y\right\vert ^{n}}u(x,t)u(y,t)dy\right)  dx.
\]
Thus
\[
\frac{d}{dt}\int_{\Omega}u\left\vert x\right\vert ^{2}dx = 4M - \frac{\chi}{\sigma_{n}}\int_{\Omega\times\Omega}\left(  \left\vert x\right\vert ^{p-2}x-\left\vert y\right\vert ^{p-2}y\right)  \cdot\frac{x-y}{\left\vert x-y\right\vert ^{n}}u(x,t)u(y,t)dydx.
\]
Using the inequality (cf. \cite{Neta})
\[
\left(  \left\vert x\right\vert ^{p-2}x-\left\vert y\right\vert ^{p-2}y\right)  \cdot\left(  x-y\right)  \geq 2^{2-p}\left\vert x-y\right\vert ^{p},
\]
for all $x,y\in\mathbb{R}^{n}$, and $p\geq2$, (cf. \cite{Neta})$,$ we get
\begin{equation}
\frac{d}{dt}\int_{\Omega}u\left\vert x\right\vert ^{2}dx \leq 4M - \frac{2^{2-p}\chi}{n\sigma_{n}}\int_{\Omega\times\Omega}\left\vert x-y\right\vert ^{p-n}u(x,t)u(y,t)dydx.
\label{th}
\end{equation}
We consider separately now two cases: $p=n$ and $2\leq p<n.$ Firstly, when $p=n,$ we obtain from (\ref{th})
\begin{align}
\frac{d}{dt}\int_{\Omega}u\left\vert x\right\vert ^{2}dx \leq & 4M - \frac{2^{2-n}\chi}{n\sigma_{n}}\int_{\Omega\times\Omega}u(x,t)u(y,t)dydx \\
 = & 4M - \frac{2^{2-n}\chi}{n\sigma_{n}}M^{2}.
\end{align}
Then, we conclude that we have blow-up if $4 < 2^{2-n}\chi M/n\sigma_{n}$ or equivalently
\[
M > \frac{2^{n}n\sigma_{n}}{\chi}.
\]
For $2\leq p<n,$ using Lemma \ref{Inequality2} (See \cite[Lemma 3.2.]{Biler}), we estimate the last integral as
\[
\int_{\Omega\times\Omega}\left\vert x-y\right\vert ^{p-n}u(x,t)u(y,t)dydx \geq M^{2+\left(  n-p\right) /2}\left(  2m(t)\right)  ^{\left(  p-n\right) /2}.
\]
Thus, we get
\begin{align*}
\frac{d}{dt}m(t) \leq & 4M - \frac{2^{2-\left(  p+n\right) /2}\chi}{n\left\vert B_{1}(0)\right\vert }M^{2+\left(  n-p\right) /2}\left(m(t)\right)  ^{\left(  p-n\right) /2} \\
 = & f(m(t)),
\end{align*}
and we let $f(m(0))<0,$ or equivalently
\begin{equation}
m(0)<\left(  \frac{\chi}{2^{\left(  p+n\right) /2}n\left\vert B_{1}(0)\right\vert }\right)  ^{2/\left(  n-p\right) }M^{\left(  n-p+2\right) /\left(  n-p\right) }=:CM^{\left(  n-p+2\right) /\left(  n-p\right) }.
\label{CB}
\end{equation}
Taking into account that the condition (\ref{CB}) implies that $m(t)$ is decreasing for $t$ small enough and the fact that $f$ is an increasing function of $m$, we conclude that the right-hand side is always negative and bounded away from $f(m(0))<0.$ It follows from (\ref{CB}) that $m(t)$ will become negative in a finite amount of time. On the other hand, $m(t)$ remains always positive due to the nonnegativity of the variable $u$. This contradiction implies $T_{\max}<\infty.$
\end{proof}

\bigskip

\begin{lemma}
\label{Inequality2}
Let for a density $0\leq u\in L^{1}(\mathbb{R}^{n},(1+\left\vert x\right\vert ^{2})dx)$ the moment and mass be defined by $m=\int u(x)\left\vert x\right\vert ^{2}dx$ and $M=\int u(x)dx,$ respectively. Then for the integral
\[
J=\int_{\mathbb{R}^{n}\times\mathbb{R}^{n}}u(x)u(y)\left\vert x-y\right\vert ^{p-n}dydx,
\]
with $p\leq n,$ the inequality
\begin{equation}
M^{2+\left(  n-p\right) /2}\leq J\left(  2m\right)  ^{\left(  n-p\right) /2},
\label{In}
\end{equation}
holds.
\end{lemma}

\begin{proof}
Using the Holder inequality, we have that
\begin{align*}
M^{2} = & \int_{\mathbb{R}^{n}\times\mathbb{R}^{n}}u(x)u(y)dxdy \\
 \leq & \left(  \int_{\mathbb{R}^{n}\times\mathbb{R}^{n}}u(x)u(y)\left\vert x-y\right\vert ^{2}dxdy\right)  ^{1-\frac{2}{n-p+2}}\left(  \int_{\mathbb{R}^{n}\times\mathbb{R}^{n}}u(x)u(y)\left\vert x-y\right\vert ^{p-n}dxdy\right)  ^{\frac{2}{n-p+2}} \\
 = & \left(  \int_{\mathbb{R}^{n}\times\mathbb{R}^{n}}u(x)u(y)\left(  \left\vert x\right\vert ^{2}+\left\vert y\right\vert ^{2}-2x\cdot y\right)  dxdy\right)  ^{1-\frac{2}{n-p+2}}J^{\frac{2}{n-p+2}} \\
 \leq & \left(  2Mm-2\left\vert \int_{\mathbb{R}^{n}}xu(x)dx\right\vert ^{2}\right)  ^{1-\frac{2}{n-p+2}}J^{\frac{2}{n-p+2}}.
\end{align*}
which implies (\ref{In}).
\end{proof}

\section{Global existence for the case $\chi(x)\propto\left\vert x\right\vert ^{n-2}$}

Throughout this section, we assume that
\[
\Omega=\left\{  x\in\boldsymbol{R}^{n}||x\mid<L\right\}  ,n=2,3,4,5,\ldots,0<L<\infty.
\]

We discuss in this section the global existence of radially symmetric densities $u(x,t)=u(|x|,t)$ in the ball $B(0,R)\subset\mathbb{R}^{n}$ satisfying the system
\begin{equation}
\begin{array}
[c]{cc}%
\partial_{t}u=\Delta u-\chi\nabla\cdot(\left\vert x\right\vert ^{n-2}u\nabla v), & x\in\Omega,t>0,\\
-\Delta v=u, & x\in\Omega,t>0,\\
\frac{\partial u}{\partial n}=\frac{\partial v}{\partial n}=0, & x\in\partial\Omega,t>0,\\
u(x,0)=u_{0}(x)\geq0, & x\in\Omega.
\end{array}
\label{S}
\end{equation}

The following proposition and lemmas are shown for solutions to the case when $\chi(x)$ remains constant for every $x$ in \cite{Nagai}. However, by using a similar argument as the one in \cite{Nagai}, we can show the following lemma for solutions to (\ref{S}). Hence, here we omit the proofs.

\begin{proposition}
The system (\ref{S}) has a unique classical solution $u$ in $\Omega\times\left( 0,T_{\max}\right)$. Moreover, $u$ is positive in $\bar{\Omega}\times\left( 0,T_{\max}\right)$.
\end{proposition}

The maximal existence time $T_{\max}$ of the classical solution is positive or infinite.

\begin{lemma}
\label{L1}
Let $u$ be a solution to (\ref{S}). If $T_{\max}<\infty$, then $u$ satisfies that
\[
\lim_{t\rightarrow T_{\max}}\Vert u(\cdot,t)\Vert_{\infty}=\infty.
\]
\end{lemma}

\begin{lemma}
\label{L2}
Let $u$ be a solution to (\ref{S}) and $u_{0}\in L^{1}(\Omega)\cap L^{\infty}(\Omega)$. Suppose that
\[
\sup_{0<t<T_{\max}}\Vert\nabla v(\cdot,t)\Vert_{\infty}<\infty.
\]
Then it holds that
\[
\sup_{0<t<T_{\max}}\Vert u(\cdot,t)\Vert_{\infty}<\infty.
\]
\end{lemma}

We assume that $u_{0}$ is radial and that $\Omega$ is a bounded open ball.

\begin{theorem}
Let $\Omega=\left\{  x\in\boldsymbol{R}^{N}||x\mid<L\right\}  ,0<L<\infty ,N\geq2$ and $u_{0}\in L^{1}(\Omega)\cap L^{\infty}(\Omega)$. Assume that $\int_{\Omega}u_{0}dx<\frac{2n\sigma_{n}}{\chi}$, then the solution $u$ to (\ref{S}) exists globally in time and satisfies $\sup_{t>0}\Vert u(\cdot ,t)\Vert_{\infty}<\infty.$
\end{theorem}

\begin{proof}
We define the cumulative mass $M(r,t)$
\[
M(r,t):=\int_{B(0,r)}u(x,t)dx=\sigma_{n}\int_{0}^{r}u(\rho,t)\rho^{n-1}d\rho.
\]
It follows that $M(r,t)$ satisfies
\begin{equation}
\begin{array}
[c]{cc}%
M_{t}=M_{rr}-(n-1)r^{-1}M_{r}+\chi\sigma_{n}^{-1}r^{-1}MM_{r}, & 0<r<L,0<t<T_{\max},\\
M(0,t)=0,M(L,t)=\theta, & 0<t<T_{\max},\\
M(r,0)=\sigma_{n}\int_{0}^{r}u_{0}(\rho)\rho^{n-1}d\rho, & 0\leq r\leq L.
\end{array}
\label{SM}
\end{equation}
Next, consider the following ODE
\[
0=\overline{M}_{rr}-(n-1)r^{-1}\overline{M}_{r}+\chi\sigma_{n}^{-1}r^{-1}\overline{M}\overline{M}_{r},
\]
Then
\begin{align*}
0 = & \int_{0}^{r}\rho\overline{M}_{\rho\rho}d\rho - (n-1)\int_{0}^{r}\overline{M}_{\rho}d\rho + \chi\sigma_{n}^{-1}\int_{0}^{r}\overline{M}\overline{M}_{\rho}d\rho \\
 = & r\overline{M}_{r} - \overline{M} - (n-1)\overline{M} + \chi\left( 2\sigma_{n}\right)  ^{-1}\overline{M}^{2} \\
 = & r\overline{M}_{r} - n\overline{M} + \chi\left( 2\sigma_{n}\right) ^{-1}\overline{M}^{2}.
\end{align*}
or equivalently
\[
r\frac{d\overline{M}}{dr} = n\overline{M} - \chi\left( 2\sigma_{n}\right) ^{-1}\overline{M}^{2}=\overline{M}\left(  n - \chi\left( 2\sigma_{n}\right) ^{-1}\overline{M}\right).
\]
Separation of variables leads to
\[
\overline{M}(r)=\frac{2n\sigma_{n}}{\chi}\frac{kr^{n}}{1+kr^{n}}<\frac{2n\sigma_{n}}{\chi}.
\]
where $k$ is the constant of integration. We note that $k$ can be chosen sufficiently large such that
\[
\theta<\overline{M}(L) \text{ and } M(r,0)\leq Cr^{n}\leq\overline{M}(r) \text{ for } 0\leq r\leq L,
\]
where $C=n^{-1}\sigma_{n}\left\Vert u_{0}\right\Vert _{L^{\infty}}.$ By the comparison theorem
\[
M(r,t)\leq\overline{M}(r) \text{ for } 0\leq r\leq L, 0\leq t<T_{\max}.
\]
Consequently,
\begin{align*}
\left\vert \nabla v(x,t)\right\vert = & \left\vert \partial_{r}v(r,t)\right\vert =\sigma_{n}^{-1}r^{1-n}M(r,t) \\
 \leq & \sigma_{n}^{-1}r^{1-n}\overline{M}(r)\leq\frac{2n\sigma_{n}}{\chi}\frac{kr}{1+kr^{n}}\leq\frac{2n\sigma_{n}kL}{\chi}.
\end{align*}
From this, Lemmas \ref{L1} and \ref{L2}, we get this theorem.
\end{proof}


\begin{thebibliography}{99}
\bibitem {BilerLocal} Biler, P. (1991). Existence and asymptotics of solutions for a parabolic-elliptic system with nonlinear no-flux boundary conditions. Universit\'{e} de Paris-sud, D\'{e}partement de math\'{e}matiques.

\bibitem {Biler} Biler, P., \& Espejo Arenas, E., \& Guerra, I. (2013). Blowup in higher dimensional two species chemotactic systems. Communications on Pure and Applied Analysis, 12, 89.

\bibitem {Cohen} Cohen, D. J., James Nelson, W., \& Maharbiz, M. M. (2014). Galvanotactic control of collective cell migration in epithelial monolayers. Nature materials, 13(4), 409-417.

\bibitem {Charras} Charras, G., \& Sahai, E. (2014). Physical influences of the extracellular environment on cell migration. Nature reviews Molecular cell biology, 15(12), 813-824.

\bibitem {Cuchi} Cucchi, A., Etchegaray, C., Meunier, N., Navoret, L., \& Sabbagh, L. (2020). Cell migration in complex environments: chemotaxis and topographical obstacles. ESAIM: Proceedings and Surveys, 67, 191-209.

\bibitem {Nagai} Nagai, T. (1995). Blow-up of radially symmetric solutions to a chemotaxis system. Adv. Math. Sci. Appl., 5, 581.

\bibitem {Neta} Neta, B. (1980). On three inequalities. \textit{Computers \& Mathematics with Applications, 6}(3), 301-304.

\bibitem {Petrie} Petrie, R. J., Doyle, A. D., \& Yamada, K. M. (2009). Random versus directionally persistent cell migration. Nature reviews Molecular cell biology, 10(8), 538-549.

\bibitem {Poff} Poff, K. L., \& Skokut, M. (1977). Thermotaxis by pseudoplasmodia of Dictyostelium discoideum. Proceedings of the National Academy of Sciences, 74(5), 2007-2010.

\bibitem {Pohoazev} Pohozaev, S. (1965). Eigenfunctions of the equation $u+\lambda f(u)=0$. In Soviet Math. Dokl (Vol. 6, pp. 1408-1411).

\bibitem {Rodriguez} Lara Rodriguez, L., \& Schneider, I. C. (2013). Directed cell migration in multi-cue environments. Integrative biology, 5(11), 1306-1323.

\bibitem {Senba} T. Senba (2005). Blowup behavior of solutions to the rescaled J\"{a}ger-Luckhaus system. Funkcialaj Ekvacioj, 48 247-271

\bibitem {Van} Van Haastert, P. J., \& Devreotes, P. N. (2004). Chemotaxis: signalling the way forward. Nature reviews Molecular cell biology, 5(8), 626-634.

\bibitem {Wondergem} Wondergem, J. A., Mytiliniou, M., Wit, F. C. D., Reuvers, T. G., Holcman, D., \& Heinrich, D. (2019). Chemotaxis and topotaxis add vectorially for amoeboid cell migration. bioRxiv, 735779.
\end{thebibliography}
\end{document}